\documentclass{amsart}
\usepackage{ifpdf}
\usepackage{amsmath, amsthm, amsfonts, amssymb, amscd}
\usepackage[active]{srcltx}
\usepackage{graphicx}

\newcommand{\ine}{\operatorname{Ine}}
\newcommand{\ess}{\operatorname{Ess}}

\newcommand{\abs}[1]{\left\lvert{#1}\right\rvert}
\newcommand{\norm}[1]{\left\|{#1}\right\|}

\DeclareMathOperator{\bd}{\partial}
\DeclareMathOperator{\cl}{cl}

\newcommand{\ol}{\overline}
\renewcommand{\hat}{\widehat}

\newcommand{\R}{\mathbb{R}}\newcommand{\N}{\mathbb{N}}
\newcommand{\Z}{\mathbb{Z}}
\newcommand{\T}{\mathbb{T}}

\newcommand{\sm}{\setminus}

\newcommand{\ie}{i.e.\ }

\newtheorem{theorem}{Theorem}
\newtheorem{corollary}[theorem]{Corollary}
\newtheorem{lemma}[theorem]{Lemma}
\newtheorem{proposition}[theorem]{Proposition}

\newtheorem{claim}{Claim}
\newtheorem*{theorem*}{Theorem}

\theoremstyle{definition}

\theoremstyle{remark}

\begin{document}
\title[Rotation sets with nonempty interior and transitivity]{Rotation sets with nonempty interior\\ and transitivity in the universal covering}
\author{Nancy Guelman}
\address{Nancy Guelman. IMERL, Facultad de Ingenier\'\i a, Universidad de la Rep\'ublica, C.C. 30, Montevideo, Uruguay}
\email{nguelman@fing.edu.uy}
\author{Andres Koropecki}
\address{Andres Koropecki. Universidade Federal Fluminense, Instituto de Matem\'atica e Estat\'\i stica, Rua M\'ario Santos Braga S/N, 24020-140 Niteroi, RJ, Brasil}
\email{ak@id.uff.br}
\author{Fabio Armando Tal}
\address{Fabio Armando Tal. Instituto de Matem\'atica e Estat\'\i stica, Universidade de S\~ao Paulo, Rua do Mat\~ao 1010, Cidade Universit\'aria, 05508-090 S\~ao Paulo, SP, Brazil}
\email{fabiotal@ime.usp.br}
\thanks{The first author was supported by Grupo de investigaci\'on de Sistemas Din\'amicos CSIC 618, Universidad de la Rep\'ublica, Uruguay.
The second author was supported by CNPq-Brasil. The third author was partially supported by CNPq-Brasil and FAPESP}

\begin{abstract}
Let $f$ be a transitive homeomorphism of the two-dimensional torus in the homotopy class of the identity. We show that a lift of $f$ to the universal covering is transitive if and only if the rotation set of the lift contains the origin in its interior.
\end{abstract}

\maketitle

\section{Introduction}
 Given an homeomorphism $f$ of the torus $\T^2=\R^2/\Z^2$ which is in the homotopy class of the identity, and a lift $\hat{f}\colon \R^2\to \R^2$ of $f$ to the universal covering, one can associate to $\hat{f}$ its rotation set $\rho(\hat f)$,  a convex and compact subset of the plane consisting of all the limit points of sequences of the form $\lim_{k\to\infty}(\hat f^{n_k}(x_k)-x_k)/n_k$, with $n_k\to \infty$ and $x_k\in \R^2$. This definition was introduced by Misiurewicz and Ziemian in \cite{MZ89} as a generalization of the rotation number for circle homeomorphisms, and has proved to be a useful tool in the study of the dynamics of these homeomorphisms.
In particular, when the rotation set has nonempty interior it has been shown that $f$ exhibits very rich dynamics, with an abundance of periodic points and positive topological entropy \cite{F89, LM91}. But this complex behavior is not restricted to rotation sets with nonempty interior: there are several examples of homeomorphisms with rich dynamical properties such that their rotation set is a singleton.

In this work we are concerned with the interplay between transitivity of $f$, transitivity on the universal covering space and rotation sets. Our motivation in studying transitivity on the universal covering space is to understand when it is possible for the dynamical system to exhibit this form of extreme transitivity, where there exist trajectories that are not only dense on the surface, but also explore all possible loops and directions. This is a strictly more stringent concept, as ergodic rotations of the $2$-torus do not lift to transitive homeomorphisms of the plane. There have been some prior works concerning transitivity of the lifted dynamics of surface homeomorphisms. 
In \cite{Boyland}, Boyland defines $H_1$-transitivity for a surface homeomorphism as the property of having an iterate which lifts to the universal Abelian covering as a transitive homeomorphism, and he characterized $H_1$-transitivity for rel pseudo-Anosov maps. 

 In \cite{guelman-et-al} it was proven that a $C^{1+\alpha}$ diffeomorphism of $\T^2$ homotopic to the identity and with positive entropy has a rotation set with nonempty interior if $f$ has a transitive lift to a suitable covering of $\T^2$. In \cite{AZT1} and \cite{AZT2}, homeomorphisms of the annulus with a transitive lift were studied, and it was shown that if there are no fixed points in the boundary then the rotation set of the lift contains the origin in its interior.
 
 More relevant to this work is the result from \cite{Ta1}, where it is shown that if $f$ is a homeomorphism of $\T^2$ homotopic to the identity which has a transitive lift $\hat{f}$ to $\R^2$, then the origin lies in the interior of $\rho(\hat f)$. Our main result is the converse of this fact:
\begin{theorem}\label{th:transitivoemcima}
Let $f$ be a transitive homeomorphism of $\T^2$ homotopic to the identity, and let $\hat f$ be a lift of $f$ to $\R^2$. If $(0,0)$ belongs to the interior of $\rho(\hat f)$ then $\hat f$ is transitive.
\end{theorem}

Note that, since the fundamental group of $\T^2$ is abelian, the notion of $H_1$-transitivity (as defined in \cite{Boyland}) in the case of a homeomorphism $f$ of $\T^2$  is equivalent to saying that some iterate of $f$ has a lift to $\R^2$ which is transitive.
As an immediate consequence of theorem \ref{th:transitivoemcima} together with the main result of \cite{Ta1}, we have a characterization of $H_1$-transitive homeomorphisms of $\T^2$ in the homotopy class of the identity: 

\begin{corollary} A homeomorphism of $\T^2$ in the homotopy class of the identity is $H_1$-transitive if and only if it has (a lift with) a rotation set with nonempty interior.
\end{corollary}

The key starting point for the proof of theorem \ref{th:transitivoemcima} is a recent result from \cite{KoTa} that rules out the existence of `unbounded' periodic topological disks. After this, the theorem is reduced to a technical result, theorem \ref{th:maintheorem}, the proof of which is obtained from geometric and combinatorial arguments relying on a well-known theorem of Franks about realization of rational rotation vectors by periodic points \cite{F89} and some classic facts from Brouwer theory. The next section states these preliminary facts and introduces the necessary notation. In section 3, theorem \ref{th:maintheorem} is stated, and its proof is presented after using it to prove theorem \ref{th:transitivoemcima}.

\section{Notations and preliminary results}

We denote by $\pi\colon \R^2\to \T^2$ the universal covering of $\T^2=\R^2/\Z^2$.
Given a point $x\in \R^2$, we denote by $(x)_1$ the projection of $x$ onto the first coordinate, by $(x)_2$ its projection onto the second coordinate, and $$\norm{x}_{\infty}=\max\{\abs{(x)_1},\abs{(x)_2}\}.$$
The following definitions will simplify the notation in the proofs ahead. For $v\in\Z^2$ and $L\in\N$, define
$$S(v,L)=\{ w\in\Z^2 : \norm{w-v}_{\infty}\le L\},$$
and
$$\Sigma(M)=\{(i,v)\in \Z\times\Z^2 : \abs{i} \le M, \norm{v}_{\infty}\le M\}.$$
Given $v\in\R^2$, we denote by $T_v:\R^2\to\R^2$ the translation $x\mapsto x+v$. If $\gamma:[0,1]\to\R^2$ is curve, we denote its image by $[\gamma]$. We say that $\gamma$ is a \emph{$T_v$-translation arc} if it joins a point $z=\gamma(0)$ to its image $T_v(z)=\gamma(1)$, and $[\gamma]\cap T_v([\gamma])=\gamma(1)$.

The following lemma is a direct consequence of lemma 3.1 of \cite{Brown}.

\begin{lemma}\label{lm:disjunto1vezdisjuntosempre}
Let $v\in\R^2$ be a nonzero vector and let $K\subset\R^2$ be an arcwise connected set such that $K\cap T_v(K)=\emptyset$. Then $K\cap T^{i}_v(K)=\emptyset$ for all $i\in\Z$ with $i\neq 0$.
\end{lemma}

And the next result follows from theorem 4.6 of \cite{Brown}.

\begin{lemma}\label{lm:Brouwerline}
Let $v\in\R^2\sm \{(0,0)\}$ and let $K$ be an arcwise connected subset of $\R^2$ such that $K\cap T_v(K)=\emptyset$, and $\alpha$ a $T_v$-translation arc disjoint from $K$. Then, either
$$K\cap\bigcup_{i\in\N}T^{i}_v[\alpha]=\emptyset \quad \text{ or } \quad K\cap\bigcup_{i\in\N}T^{-i}_v[\alpha]=\emptyset.$$
\end{lemma}

We will also use the next theorem from \cite{F89}.
\begin{theorem}[Franks] \label{th:Franks} Let $\hat{f}$ be a lift to $\R^2$ of a homeomorphism of $\T^2$ homotopic to the identity. If $w\in \Z^2$ and $q\in \N$ are such that $w/q$ lies in the interior of $\rho(\hat{f})$, then there is $\hat{x}\in \R^2$ such that $\hat f^q(\hat x)=\hat x+w.$
\end{theorem}

Let us recall some terminology used in \cite{KoTa}. An open set $U\subset \T^2$ is \emph{inessential} if every loop contained in $U$ is homotopically trivial in $\T^2$. An arbitrary subset of $\T^2$ is called inessential if it has an inessential neighborhood, and \emph{essential} otherwise. If $A$ is an inessential subset of $\T^2$, and $U$ is an open neighborhood of $\T^2\setminus A$, then $\T^2\setminus U$ is a closed inessential set. This implies that every homotopy class of loops of $\T^2$ is represented by some loop contained in $U$, and for this reason any subset of $\T^2$ with an inessential complement is called a \emph{fully essential} set.

If $f\colon \T^2\to\T^2$ is a homeomorphism, we say that a point $x\in \T^2$ is an \emph{inessential point} for $f$ if there exists $\epsilon>0$ such that the set $\bigcup_{i\in\Z} f^{i}(B_{\epsilon}(x))$ is inessential. The set of all inessential points of $f$ is denoted by $\ine(f)$. Any point in the set $\ess(f)=\T^2\sm \ine(f)$ is called an \emph{essential point} for $f$.  It follows from the definitions that $\ine(f)$ is open and invariant, while $\ess(f)$ is closed and invariant.

Given an open connected set $U \subset \T^2$, we denote by $\mathcal{D}(U)$ the diameter in $\R^2$ of any connected component of $\pi^{-1}(U)$. If $\mathcal{D}(U) < \infty$, then we say that $U$ is \emph{bounded}.

Suppose that $f\colon \R^2\to \R^2$ is a homeomorphism homotopic to the identity. We need the following
\begin{theorem}[\cite{KoTa}, Theorem C]\label{th:KoTa}
If $f$ is nonwandering and has (a lift with) a rotation set with nonempty interior, then $\ine(f)$ is a disjoint union of periodic simply connected sets which are bounded.
\end{theorem}

In particular, when $f$ is transitive one may conclude $\ine(f)$ is empty:

\begin{lemma}\label{lm:essentialisall}
Suppose that $f$ is transitive and has (a lift with) a rotation set with nonempty interior. Then $\ess(f)=\T^2$.
\end{lemma}

\proof
Suppose for contradiction that $\ine(f)$ is nonempty. Then there is some point $x\in \ine(f)$ with a dense orbit. Let $D$ be the connected component of $\ine(f)$ that contains $x$. By lemma \ref{th:KoTa}, we have that $D$ is a periodic simply connected open and bounded set. Thus any connected component $\hat{D}$ of $\pi^{-1}(D)$ is bounded. 

Let $k$ be such that $f^k(D)=D$ and let $\hat{f}$ be a lift of $f$. Then $\hat{f}^k(\hat{D})=\hat{D}+v$ for some $v\in \Z^2$, and so if $\hat{g}=T_v^{-1}\hat{f}^k$ we have that $\hat{g}(\hat{D})=\hat{D}$. The set $\hat{U}=\bigcup_{i=0}^{k-1}\hat{f}^i(\hat{D})$ is then bounded $\hat{g}$-invariant (because $\hat{g}$ commutes with $\hat{f}$), and therefore $\cl(\hat{U})$ is also bounded and $\hat{g}$-invariant. 
In particular all points in $\cl(\hat{U})$ have a bounded $\hat{g}$-orbit. Since $U=\pi(\hat{U})$ is the $f$-orbit of $D$, which contains the (dense) orbit of $x$, it follows that $\cl(U)=\T^2$. The boundedness of $\hat{U}$ then implies that $\pi(\cl(\hat{U})) = \cl(U)=\T^2$. 

It follows from these facts that every point of $\R^2$ has a bounded $g$-orbit, from which we conclude that $\rho(\hat{g})=\{(0,0)\}$. But it follows from the definition of rotation set that $\rho(\hat{g})=\rho(T_v^{-1}\hat{f}^k) = k\rho(\hat{f})-v$ (see \cite{MZ89}), contradicting the fact that $\rho(\hat{f})$ has nonempty interior.
\endproof

\section{Proof of Theorem \ref{th:transitivoemcima}}

Theorem \ref{th:transitivoemcima} is a consequence of the following technical result.
\begin{theorem}\label{th:maintheorem}
Let $f\colon \T^2\to \T^2$ be a homeomorphism homotopic to the identity, and $\hat f$ a lift of $f$ such that $(0,0)$ belongs to the interior of $\rho(\hat f).$
Let $O\subset \R^2$ be an open connected set such that $\ol{\pi(O)}$ is inessential and $\bigcup_{n\in\Z} f^n(\pi(O))$ is fully essential. Then, for every $w\in\Z^2$ there exists $n\in \N$ such that $\hat f^n(O)\cap T_w(O)\neq \emptyset$.
\end{theorem}

\begin{proof}[Proof of theorem \ref{th:transitivoemcima} assuming theorem \ref{th:maintheorem}]
Let $O_1, O_2$ be two open sets of the plane. Since $f$ is transitive, there exists $n_0$ such that $f^{n_0}(\pi(O_1))$ intersects $\pi(O_2),$ which implies that there exists $w\in\Z^2$ such that $\hat f^{n_0}(O_1)\cap T_w(O_2)\not=\emptyset.$ Let $\hat x\in\R^2$ and $0<\varepsilon<\frac{1}{2}$ be such that $B_{\varepsilon}(\hat x)\subset  \hat f^{n_0}(O_1)\cap T_w(O_2)$. Note that, as $\varepsilon<\frac{1}{2}$, the set $\pi(B_{\varepsilon}(\hat x))$ is inessential.

Since $f$ is transitive and $\rho(\hat f)$ has interior, lemma \ref{lm:essentialisall} implies that $\ess(f)=\T^2$ and, in particular, 
$ U_{\varepsilon}(x)=\bigcup_{i\in\Z} f^{i}(B_{\varepsilon}(x))$ is an essential invariant open set and, since $f$ is nonwandering, the connected component of $\hat x$ in $U_{\varepsilon}(x)$ is periodic. In $\T^2$, an essential connected  periodic open set is either fully essential or contained in a periodic set homeomorphic to an essential annulus, in which case the rotation set is contained in a line segment. Since $\rho(\hat f)$ has nonempty interior, we conclude that $U_{\varepsilon}(x)$ is fully essential.

Therefore $B_{\varepsilon}(\hat x)$ satisfies that  the hypotheses of theorem \ref{th:maintheorem}. Thus there exists $n_1$ such that $\hat f^{n_1}(B_{\varepsilon}(\hat x))$ intersects $T_{-w}(B_{\varepsilon}(\hat x)).$ But this implies that $\hat f^{n_0+n_1}(O_1)$ intersects $T_{-w}(T_{w}(O_2))=O_2,$ completing the proof.
\end{proof}

\subsection{Proof of theorem \ref{th:maintheorem}}
 
We begin by fixing both $\overline{w}\in\Z^2$ and an open connected set $O\subset\R^2$ such that $\bigcup_{n\in\Z} f^n(\pi(O))$ is fully essential and $\ol{\pi(O)}$ is inessential. Our aim is to show that $\hat{f}^n(O)\cap T_{\ol{w}}(O)\neq \emptyset$ for some $n\in \N$. 

Since we are assuming that $\rho(\hat{f})$ contains $(0,0)$ in its interior, we may choose $\delta>0$ (which is fixed from now on) such that $B_{\delta}((0,0))\subset \rho(\hat f)$.

\begin{proposition}\label{pr:cercadominiofundamental}
There exists a positive integer $M$ and a compact set $K$ such that $[0,1]^2$ is contained in a bounded connected component of $\R^2\sm K$ and $$K \subset\bigcup_{(i,v)\in\Sigma(M)}
T_v(\hat f^{i}(O))$$
\end{proposition}

\proof
Since $\bigcup_{n\in\Z} f^n(\pi(O))$ is fully essential, its preimage by $\pi$, which we denote by $\hat U$, is a connected open set invariant by $\Z^2$ translates. Therefore, given a point $\hat y$ in $\hat U,$ there exist two connected arcs $\alpha$ and $\beta$ in $\hat U$ such that $\alpha$ connects $\hat y$ to $\hat y+(1,0)$ and $\beta$ connects $\hat y$ to $\hat y+(0,1).$  Let 
$$\Gamma_{\alpha}=\bigcup_{i=-\infty}^{\infty} T^{i}_{(1,0)}[\alpha],\quad \Gamma_{\beta}=\bigcup_{i=-\infty}^{\infty} T^{i}_{(0,1)}[\beta]$$
and note that, as $\hat U$ is $\Z^2$ invariant, then all integer translates of $\Gamma_{\alpha}$ and $\Gamma_{\beta}$ are contained in $\hat U.$

Fix an integer $R> \max \{\norm{x} : x\in [\alpha]\cup[\beta]\}$. Then, since 
$$\max \{(x)_2 : x \in \Gamma_{\alpha}\}=\max \{(x)_2 : x\in[\alpha]\}<R$$ and $\min \{(x)_2 : x \in \Gamma_{\alpha}\}>-R,$ it follows that $\R^2\sm \Gamma_{\alpha}$ has at least two connected components, one containing the semi-plane $\{ x : (x)_2\ge R\}$ and another containing $\{x : (x)_2 \le -R\}.$
Likewise, $\R^2\sm \Gamma_{\beta}$ has at least two connected components, one containing $\{ x : (x)_1\ge R\}$ and another one containing $\{x : (x)_1 \le -R\}.$

Now let
$$F= (\Gamma_\alpha-(0,R))\cup (\Gamma_\alpha+(0,R+1))\cup(\Gamma_\beta-(R,0))\cup(\Gamma_\beta+(R+1,0)),$$
and note that $F\subset \hat U$ and $\R^2\sm F$ has a connected component $W$ that contains $[0,1]^2$ and is contained in
$[-2R, 2R+1]\times [-2R, 2R+1]$ (see figure \ref{fig:1}). 
\begin{figure}[ht]
\begin{center}\includegraphics[height=6cm]{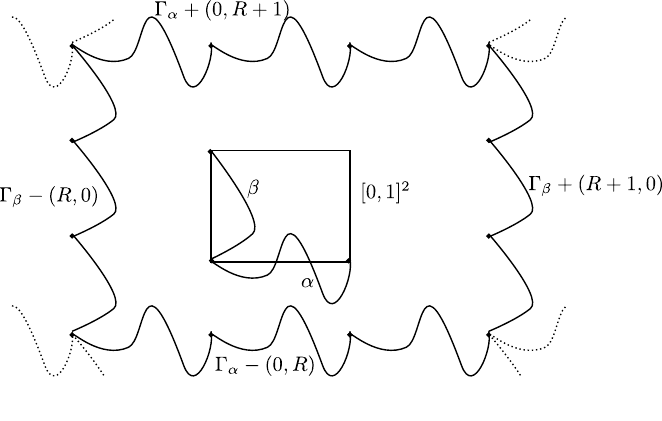}
\end{center}
\caption{}
\label{fig:1}
\end{figure}

Let $K = \bd W$. Then $K$ is a compact subset of $\hat U$.
Since $\hat{U} = \bigcup_{v\in\Z^2}\bigcup_{i\in\Z} T_v(\hat f^{i}(O))$ is an open cover of $K$, choosing a finite subcover we conclude the existence of $M$.
\endproof

\begin{proposition}\label{pr:fundamentaldomainmeets}
For every $w\in\Z^2,$ if $n>\frac{\norm{w}}{\delta},$ then $\hat f^n([0,1]^2)$ intersects $T_w([0,1]^2).$
\end{proposition}
\proof
By our choice of $\delta$ (at the beginning of this section), $w/n$ is in the interior of $\rho(\hat{f})$. It follows from theorem \ref{th:Franks} that there is $\hat y\in \R^2$ such that $\hat f^n(\hat y)=\hat y + w$. Since we may choose $\hat y \in [0,1]^2$ (by using an appropriate integer translation), the proposition follows.
\endproof

\begin{proposition}\label{pr:cadawtemumtransladado}
For every $w\in\Z^2,$ and every $n>\frac{\norm{w}}{\delta},$ there exists $(j_w,v_w)$ in $\Sigma(2M)$ such that $\hat f^{n}(\hat f^{j_w}(T_{v_w}(O)))\cap T_w(O)\neq \emptyset$.
\end{proposition}

\proof
By proposition \ref{pr:fundamentaldomainmeets}, if $U$ is the connected component of $\R^2\sm K$ that contains $[0,1]^2,$ then $\hat  f^n(U)$  intersects $T_w(U)$, and since $U$ is bounded this implies that $\hat f^n(\partial U)$ intersects $T_w(\partial U).$ Since $\partial U\subset K$, it follows that $\hat f^{n}(K)\cap T_w(K)\not=\emptyset.$ If $\hat x$ is a point in this intersection, then there exist integers $j_1, j_2\in [-M,M]$ and $v_1, v_2\in \Z^2$ with $\norm{v_i}_{\infty}<M$, $i\in\{1,2\},$  such that
$\hat x$ belongs to both $\hat f^{n}(\hat f^{j_1}(T_{v_1}(O))) $ and $\hat f^{j_2}(T_{v_2}(O+w)).$ Setting $j_w=j_1-j_2$ and $v_w=v_1-v_2$ yields the result.
\endproof

\begin{lemma}\label{lm:inbetween}
Let $v\in \R^2$ be a nonzero vector, and $K_1$, $K_2$ be two arcwise connected subsets of $\R^2$ such that $T_v(K_i)\cap K_i=\emptyset$ for $i\in \{1,2\}$. Suppose there are integers $i,j$ with $ i\geq 0$, $j > 0$ such that $T_v^{-i}(K_1)\cap K_2\not=\emptyset$ and
$T_v^{j}(K_1)\cap K_2\not=\emptyset$. Then $K_1$ intersects $K_2$. Moreover, there exists a $T_v$-translation arc $\gamma$ contained in $K_1\cup K_2$ and joining a point $x\in K_1$ to $T_v(x)\in K_2$.
\end{lemma}

\proof
We may assume that $K_1$ is compact by replacing it by a compact arc contained in $K_1$ and joining some point of $T_v^{i}(K_2)$ to a point of $T_v^{-j}(K_2)$.

Let $\alpha:[0,1]\to K_2$ be a simple arc satisfying $\alpha(0)\in T^{-i}_{v}(K_1)$ and $\alpha(1)\in T^j_{v}(K_1).$ Since $[\alpha]$ and $K_1$ are compact and $v$ is not null, there exists an integer $n_0$ such that, if $\abs{n}>n_0,$ then $T^{n}_v(K_1)$ is disjoint from $[\alpha].$

Let
\begin{eqnarray}\nonumber
s_0 = \max\{t\in[0,1]: \alpha(t)\in \bigcup_{n=0}^{n_0}T^{-n}_{v}(K_1)\},\\
 s_1 = \min\{t\in[s_0,1]: \alpha(t)\in \bigcup_{n=1}^{n_0}T^n_{v}(K_1)\},\nonumber
\end{eqnarray}
and let $i_0$ and $j_0$ be integers such that $\alpha(s_0)\in T^{-i_0}_{v}(K_1)$ and $\alpha(s_1)\in T^{j_0}_{v}(K_1).$ Finally, let
$$K_3 = T^{-i_0}_{v}(K_1)\cup\alpha([s_0,s_1])\cup T^{j_0}_{v}(K_1),$$
which is a connected set.

We claim that $i_0=0$ and $j_0=1$. To prove our claim, assume that it does not hold. Then $i_0+j_0>1$. Note that by construction, $\alpha((s_0,s_1))$ is disjoint from $\bigcup_{n\in\Z}T^n_{v}(K_1).$ Since $K_1\cap T_v(K_1)=\emptyset$, it follows from lemma \ref{lm:disjunto1vezdisjuntosempre} that $K_1\cap T_v^n(K_1)=\emptyset$ for any $n \neq 0$, and so $T_v(T^{-i_0}_{v}(K_1))$ is disjoint from $T^{j_0}_{v}(K_1)$  (because $i_0+j_0\neq 1$). From these facts and from the hypotheses follows that $K_3$ is disjoint from $T_v(K_3)$, so again by lemma \ref{lm:disjunto1vezdisjuntosempre} we have that $K_3$ is disjoint from $T_v^n(K_3)$ for all $n\neq 0$. But $i_0+j_0\neq 0$, and clearly $T^{j_0+i_0}_v(K_3)$ intersects $K_3$. This contradiction shows that $i_0+j_0 = 1$, i.e.  $i_0=0$ and $j_0=1.$

Since $\alpha(s_0)\in K_2\cap T^{-i_0}_v(K_1) = K_2\cap K_1$, we have shown that $K_1$ intersects $K_2$. Note that $\alpha(s_1)\in T_{v}^{j_0}(K_1) = T_v(K_1)$, and let $\beta:[0,1]\to K_1$ be a simple arc joining $\alpha(s_1)-v$ to $\alpha(s_0)$. Since $[\beta]\subset K_1$ and $\alpha([s_0,s_1))\cap (T_v(K_1)\cup T_{-v}(K_1))=\emptyset,$ it follows that $[\beta]\cap T_v(\alpha([s_0,s_1)))= \emptyset=T_v([\beta])\cap \alpha([s_0,s_1))$. Therefore, letting $\gamma$ be the concatenation of $\beta$ with $\alpha|_{[s_0,s_1]}$ we conclude that $\gamma$ is a $T_v$-translation arc joining $x=\alpha(s_1)-v\in K_1$ to $T_v(x)\in K_2$.
\endproof

Define, for each $(j,v)\in \Z\times \Z^2,$
the sets
$$I(j,v)=\{w\in\Z^2 : \, \hat f^{j}(T_v(O))\cap T_w(O)\not=\emptyset\}.$$

\begin{proposition}\label{pr:propertiesofsetofindices} The following properties hold.
\begin{enumerate}
\item{ If $R>0$ and $n>\frac{R}{\delta}$, then
$$\left(\Z^2\cap B_R((0,0))\right)\subset\bigcup_{(j,v)\in \Sigma(2M)} I(j+n,v).$$}
\item{For any $w,v\in\Z^2$ and $ j\in\Z,\, T_w(I(j,v))= I(j, T_w(v))$.}
\item{ Let $u,v \in \Z^2$ and $j\in\Z.$ If $S(u,C)\subset I(j,v)$, then $S(u,C-\norm{v-w}_{\infty})\subset I(j,w)$.}
 \end{enumerate}
\end{proposition}

\proof Parts (1) and (2) are direct consequences of definition of $I(j,v)$ and proposition \ref{pr:cadawtemumtransladado}. To prove (3), let $ r \in S(u,C-\norm{v-w}_{\infty})$; since $ r+v-w \in S(u,C)\subset I(j,v)$, it follows from (2) that $r= T_{w-v}(r+v-w)\in I(j,T_{w-v}(v))=I(j,w)$.
\endproof

\begin{proposition}\label{pr:pintasegmento}
Let $i,k_0,k_1,j$  be integers with $k_0<k_1$, and $u\in\Z^2.$ If both $(i,k_0)$ and $(i,k_1)$ belong to $I(j,u),$ then $(i,k) \in I(j,u)$ for each integer $k\in [k_0,k_1]$. Likewise, if both $(k_0,i)$ and $(k_1,i)$ belong to $I(j,u),$ then $(k,i) \in I(j,u)$ for all integers $k\in [k_0,k_1]$.
\end{proposition}
\begin{proof}
This is a direct consequence of lemma \ref{lm:inbetween}, choosing $K_1 = O +(i,k)$, $K_2=f^{j}(T_u(O))$ and $v=(1,0)$ for the first case, and an analogous choice for the second case.
\end{proof}

\begin{lemma}\label{lm:pintasubquadrado}
Given $C_1\in\Z$ there is $C_2>0$ such that, for any $w_1\in \Z^2$, $C>C_2$, and $n>(\norm{w_1}+\sqrt{2}C)/\delta$, there exist $w_2\in\Z^2$ and $\overline j \in [-2M,2M]\cap \Z$,  such that $$S(w_2,C_1)\subset S(w_1, C)\cap I(n+\overline j, \overline w).$$
\end{lemma}

\proof
Define the auxiliary constants $L= (4M+1)^3 +1$, which is larger than the cardinality of the set $\Sigma(2M)$, and $R=2M+\norm{\overline{w}}_{\infty}.$ Let also $D= 2C_1 + 2R+1$, and let $C_2=L^L D.$

Note that if $C>C_2$ and $n$ is chosen as in the statement, proposition \ref{pr:propertiesofsetofindices} (1) implies that
$$S(w_1, C)\subset \bigcup_{(j,v)\in \Sigma(2M)}I(n+j,v).$$
Let $i_0,k_0$ be integers such that $(i_0,k_0)= w_1 - (C,C)$. Define,  for $1\le s\le L-1, \, k_s = k_0+sD$. For $1\le s \le L$ we will define $i_s$ in $\Z$ and $(j_s,v_s)\in\Sigma(2M)$ recursively satisfying the following properties:
\begin{enumerate}
\item{$i_{s-1}\le i_{s} \le i_{s-1} +\frac{C_2}{L^{s-1}}-\frac{C_2}{L^{s}}$}
\item{For any $ i_{s}\le i \le i_s+ \frac{C_2}{L^s}$, the point $(i,k_{s-1})\in I(n+j_s,v_s)$.}
\end{enumerate}

Suppose we have already defined $i_s$ and $(j_s,v_s).$ Consider the $L$ different points of the form $(i_{s,r},k_s)$ with $0\le r \le L-1$, where $i_{s,r}=i_s+r\frac{C_2}{L^{s+1}}$. 
Note that from the recursion hypothesis (1) follows that $i_{s} \leq i_0+C_2-C_2/L^{s}$, so that $$i_{s,r}\leq i_0+C_2 - C_2/L^s +(L-1)C_2/L^{s+1} = i_0+C_2 - C_2/L^{s+1} \leq i_0+C.$$ 
In particular, $(i_{s,r},k_s) \in  S(w_1, C)$ for each $r\in \{0,\dots, L-1\}$, so there exists $(j_{s,r},v_{s,r})\in \Sigma(2M)$, satisfying $(i_{s,r},k_s)\in I(n+j_{s,r},v_{s,r})$. Since $L$ is greater than the cardinality of $\Sigma(2M)$, by the pigeonhole principle there exists $r_1<r_2$ such that $(j_{s,r_1},v_{s,r_1})=(j_{s,r_2},v_{s,r_2})$. Define $i_{s+1}=i_{s,r_1}$ and $(j_{s+1},v_{s+1})=(j_{s,r_1},v_{s,r_1})$.
Since both $(i_{s,r_1},k_s)$ and $(i_{s,r_2},k_s)$ belong to $I(n+ j_{s+1},v_{s+1})$, it follows from proposition \ref{pr:pintasegmento} that $(i,k_s)\in I(n+j_{s+1},v_{s+1})$ whenever $i_{s,r_1}\le  i \le i_{s,r_2}$. Note also that $i_{s+1}\ge i_s$, and that $i_{s+1}+\frac{C_2}{L^{s+1}}\le i_{s,r_2}\le i_s+ \frac{C_2}{L^{s}}$. So, $i_{s+1}$ and $(j_{s+1},v_{s+1})$ satisfy properties (1) and (2) (see figure \ref{fig:2}).

\begin{figure}[ht]
\begin{center}\includegraphics{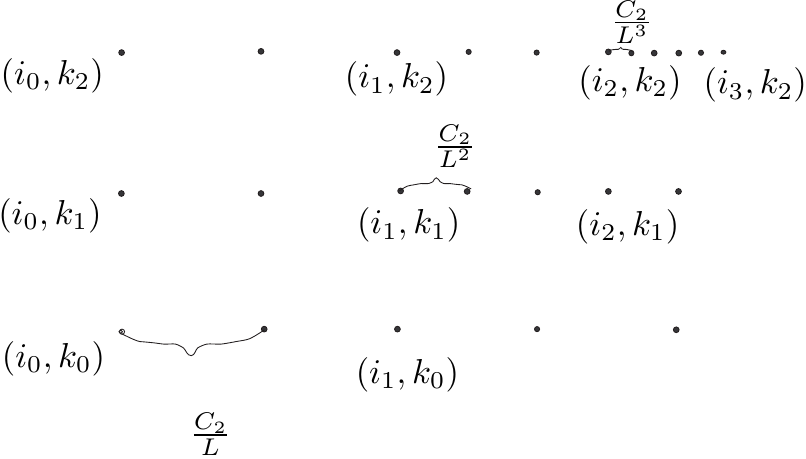}
\end{center}
\caption{}
\label{fig:2}
\end{figure}

Having defined $i_s$ and $(j_s,v_s)$ for $1\le s\le L$, we can again use the pigeonhole principle to find integers $1\le s_1<s_2\le L$ such that $(j_{s_1},v_{s_1})=(j_{s_2},v_{s_2}).$

By property (2) in the recursion, for each integer $i \in [i_{s_2}, i_{s_2}+\frac{C_2}{L^{s_2}}]$ the point $(i_{s_2}+i,k_{s_2-1})$ belongs to $I(n+j_{s_2},v_{s_2}).$

As $i_{s_1}\le i_{s_2}$ and $i_{s_2}+\frac{C_2}{L^{s_2}}< i_{s_1}+\frac{C_2}{L^{s_1}}$, then the interval $[i_{s_2},i_{s_2}+\frac{C_2}{L^{s_2}}]$ is contained in $[i_{s_1},i_{s_1}+\frac{C_2}{L^{s_1}}] $, so again by property (2), for each integer $ i \in [i_{s_2}, i_{s_2}+\frac{C_2}{L^{s_2}}]$, the point $(i_{s_2}+i,k_{s_1-1})$ belongs to $I(n+j_{s_1},v_{s_1}) = I(n+j_{s_2},v_{s_2}).$

Hence, for all $0\le i\le \frac{C_2}{L^{s_2}},$  both $(i_{s_2}+i,k_{s_1-1})$ and  $(i_{s_2}+i,k_{s_2-1})$ belong to $I(n+j_{s_2},v_{s_2})$. From proposition \ref{pr:pintasegmento} applied to the latter pair of points for each $i$ follows that
$$\{(i,k): i_{s_2}\le i \le i_{s_2}+\frac{C_2}{L^{s_2}}, k_{s_1-1}\le k \le k_{s_2-1}\}\subset I(n+j_{s_2},v_{s_2}).$$

Note that $k_{s_2-1}-k_{s_1-1}\ge D,$ and also $\frac{C_2}{L^{s_2}}\ge D.$ Therefore, if we define $w_2=(i_{s_2},k_{s_1-1})+ (R+ C_1+1,R+ C_1+1)$ then $S(w_2, R+C_1)\subset I(n+j_{s_2},v_{s_2}),$ and $S(w_2, R+C_1)\subset S(w_1, C).$

Finally, since $\norm{v_{s_2}-\ol{w}}_{\infty}\le R$, if $\ol{j}=j_{s_2}$ we have, by proposition \ref{pr:propertiesofsetofindices}, $S(w_2, C_1)\subset I(n+\ol{j},\ol{w})$.
\endproof

\begin{proposition}\label{pr:intersquadrantes}
For every $R>0$ there exists  $n_0 \in \N$ such that, if $n>n_0$, then $\hat f^n(O)$ intersects each of the four sets
$$ U_1= \{(y: (y)_1>R, (y)_2>R\},\quad U_2=\{y:(y)_1<-R, (y)_2>R\},$$
$$ U_3= \{y: (y)_1>R, (y)_2<-R\},\quad U_4=\{y:(y)_1<-R, (y)_2<-R\}$$
\end{proposition}

\proof
Let $C_1$ be such that $O\subset B_{C_1}((0,0)),$ and let $C_2=\max_{x\in \R^2} \norm{\smash{\hat f(x)-x}}$, which is finite since $f$ is homotopic to the identity.

Let $C> 2M(C_2+1)+ C_1+ R$ be an integer, and consider $w_1=(C,C), w_2=(C,-C), w_3=(-C,C)$ and $w_4=(-C,-C).$

If $n> \frac{\sqrt{2} C}{\delta}$ then, by proposition \ref{pr:cadawtemumtransladado}, for each $i\in\{1,2,3,4\}$ there exists
$(j_i, v_i)\in \Sigma(2M)$ such that $w_i\in I(n+j_i,v_i)$. This implies that $\hat f^{n+j_i}(T_{v_i}(O))$ intersects $T_{w_i}(O)\subset T_{w_i}(B_{C_1}((0,0)))=B_{C_1}(w_i)$ for each $i\in \N$. Since the definition of $C_2$ implies that, for every $x\in\R^2$, $$\norm{\smash{\hat f^{n+j_i}(x)-\hat f^{n}(x)}}\le \abs{j_i} C_2\le 2MC_2,$$ 
we conclude that $\hat f^{n}(T_{v_i}(O))$ intersects $B_{C_1+2M C_2}(w_i)$. This means that $\hat f^{n}(O)$ intersects $T_{v_i}^{-1}(B_{C_1+2M C_2}(w_i))$. Since $\norm{v_i}_{\infty}\le 2M$, it follows that
 $$T_{v_i}^{-1}(B_{C_1+2M C_2}(w_i))\subset B_{C_1+2M (C_2+1)}(w_i)\subset U_i,$$
so $\hat{f}^n(O)\cap U_i\neq \emptyset$.
\endproof

\begin{claim}\label{cl:mainclaim} 
Suppose that $\hat{f}^n(O)\cap T_{\ol{w}}(O)=\emptyset$ for all $n\in \Z$.
Let $k\in \N$ and $v\in\Z^2$ be such that $S(v,1)\subset I(k,\ol{w}).$ Then, for every sufficiently large $n,$ we have $v\in  I(n,(0,0)).$
\end{claim}

\proof
Assume, by contradiction, that there exists a sequence $(n_i)_{i\in\N}$ such that $\lim_{i\to\infty}n_i=\infty$ and such that, $v\notin I(n_i,(0,0))$ for all $i\in \N$.  Then, by proposition \ref{pr:pintasegmento}, we cannot have that both $v-(1,0)$ and $v+(1,0)$ belong to $I(n_i,(0,0))$ simultaneously, and neither can $v-(0,1)$ and $v+(0,1)$. Thus we may assume, with no loss of generality, that for infinitely many values of $i$ neither $v+(1,0)$ nor $v+(0,1)$ belong to $I(n_i,(0,0))$ (since the other cases are analogous). This means that
$$\hat f^{n_i}(O)\cap \left(T_v(O)\cup T_{v+(1,0)}(O)\cup T_{v+(0,1)}(O)\right)=\emptyset.$$
Extracting a subsequence of $(n_i)_{i\in \N}$ we may (and will) assume that the latter holds for \emph{all} $i\in \N$.

Let $K_2=\hat f^{k}(T_{\ol{w}}(O)).$  Since $S(v,1)\subset I(k,\ol{w}),$ it follows that $K_2$ intersects the open connected sets $T_v(O)$ and $T_{(1,0)}(T_v(O))$ and therefore lemma \ref{lm:inbetween} (with $K_1=T_v(O)$) implies that there exists a $T_{(1,0)}-$translation arc $\alpha$ contained in $T_v(O)\cup K_2$ joining a point $x_1\in T_v(O)$ to $T_{(1,0)}(x_1).$
Note that the assumption that $\hat{f}^n(O)$ is disjoint of $T_{\ol{w}}(O)$ for all $n\in\Z$ implies that  $K_2\cap \hat f^{n_i}(O)=\emptyset$ for all $i \in \N$. Furthermore, since $v\notin I(n_i,(0,0))$, the set $\hat f^{n_i}(O)$ is also disjoint from $T_v(O)$. Thus $[\alpha]\cap \hat f^{n_i}(O)=\emptyset$ for all $i\in \N$. 

A similar reasoning shows that  there exists a translation arc $\beta\subset T_v(O)\cup K_2,$ joining a point $x_2\in T_v(O)$ to $T_{(0,1)}(x_2)$, such that $[\beta]\cap \hat f^{n_i}(O)=\emptyset$ for all $i\in \N$.

Let $\gamma\colon [0,1]\to T_v(O)$ be an arc joining $x_1$ to $x_2$, and define 
$$\alpha^{+}=\bigcup_{i\in\N}T^i_{(1,0)}([\alpha]),\,\, \alpha^{-}=\bigcup_{i\in\N}T^{-i}_{(1,0)}([\alpha]),\,\, \beta^{+}=\bigcup_{i\in\N}T^i_{(0,1)}([\beta]),\,\, \beta^{-}=\bigcup_{i\in\N}T^{-i}_{(0,1)}([\beta]).$$

Since $\hat f^{n_i}(O)$ is disjoint from $[\alpha] \cup [\beta] \cup [\gamma]$, it follows from lemma \ref{lm:Brouwerline} that $\hat f^{n_i}(O)$ is disjoint from at least one of the following 4 connected sets:
$$F_1=\alpha^{+}\cup\beta^{+}\cup[\gamma], \, F_2=\alpha^{+}\cup\beta^{-}\cup[\gamma],$$
$$F_3=\alpha^{-}\cup\beta^{+}\cup[\gamma], \, F_4=\alpha^{-}\cup\beta^{-}\cup[\gamma].$$

In particular there is $j\in\{1,2,3,4\}$ such that $\hat{f}^{n_i}(O)$ is disjoint from $F_j$ for infinitely many values of $i\in \N$. We assume that $j=1$, since the other cases are analogous. Extracting again a subsequence of $(n_i)_{i\in \N}$, we may assume that $f^{n_{i}}(O)$ is disjoint from $F_1$ for all $i\in \N$.

  Let $R_1=\max_{y\in\alpha} \norm{y}$, $R_2=\max_{y\in\beta}\norm{y}$, $R_3=\max_{y\in\gamma}\norm{y},$ and note that $\max_{y\in\alpha^{+}}(y)_2\le R_1$,  and $ \max_{y\in\beta^{+}}(y)_1\le R_2.$

Finally, let $R=\max\{R_1,R_2,R_3\},$ so that 
$$F_1\subset \{y : \abs{(y)_1}\leq R\}\cup \{y : \abs{(y)_2}\leq R\},$$ 
and observe that the sets 
$$ U_1 = \{y : (y)_1>R, (y)_2>R\},\quad  U_2=\{y : (y)_1<-R, (y)_2>R\}$$
lie in distinct connected components of  $\R^2\sm F_1$ (see figure \ref{fig:3}).

\begin{figure}[ht]
\begin{center}\includegraphics[height=6cm]{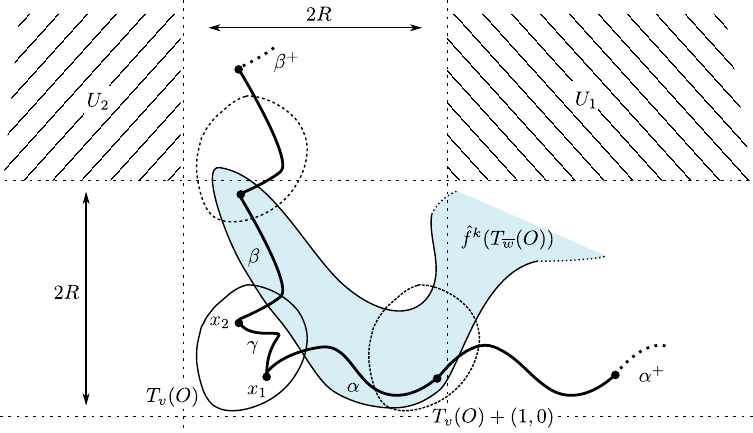}
\end{center}
\caption{}
\label{fig:3}
\end{figure}

As $\hat f^{n_i}(O)$ is connected and contained in $\R^2\sm F_1$, for each $i\in \N$ the set $\hat{f}^{n_i}(O)$ is disjoint from either $U_1$ or $U_2$. This contradicts proposition \ref{pr:intersquadrantes}, completing the proof of the claim.
\endproof

Let us denote by $R_{\ol{w}}(v)= 2\ol{w}-v=T^2_{\ol{w}-v}(v)$ the symmetric point of $v$ with respect to $\ol{w}$.

\begin{claim}\label{cl:pintaquadradoeoposto}
There exist $k_1, k_2\in \Z$ and $\ol{v}\in \Z^2$ such that $S(\ol{v},1)\subset I(k_1,\ol{w})$ and $S(R_{\ol{w}}(\ol{v}),1)\subset I(k_2,\ol{w}).$
\end{claim}

\proof

Let $C_2>0$ be such that the conclusion of lemma \ref{lm:pintasubquadrado} holds with $C_1=1$. Using again lemma \ref{lm:pintasubquadrado} but setting $C_1=C_2+1$, and $w_1=(0,0)$, there exists $C_2'>0$, $n_1\in \N$, an integer $j_1\in [-2M,2M]$, and $v_1\in \Z^2$ such that $$S(v_1,C_2+1)\subset S((0,0),C_2'+1) \cap I(n_1+j_1, \ol{w}).$$ 

Due to our choice of $C_2$, lemma \ref{lm:pintasubquadrado} applied with $C_1=1$ and $w_1=R_{\ol{w}}(v_1)$ implies that there exist $n_2\in \N$, an integer $j_2\in [-2M, 2M]$ and $v_2\in \Z^2$ such that 
$$S(v_2,1)\subset S(R_{\ol{w}}(v_1),C_2+1)\cap I(n_2+j_2,\ol{w}).$$

Let $\ol{v}=v_2$, $k_1=n_2+j_2$ and $k_2=n_1+j_1$. Then, $S(\ol{v},1)\subset I(k_1,\ol{w})$, while 
$$S(R_{\ol{w}}(\ol{v}),1) = R_{\ol{w}}(S(\ol{v},1))\subset R_{\ol{w}}(S(R_{\ol{w}}(v_1),C_2+1)) = S(v_1, C_2+1)\subset I(k_2,\ol{w}),$$
so the claim follows (see figure \ref{fig:claim2}).
\begin{figure}[ht]
\includegraphics[height=4cm]{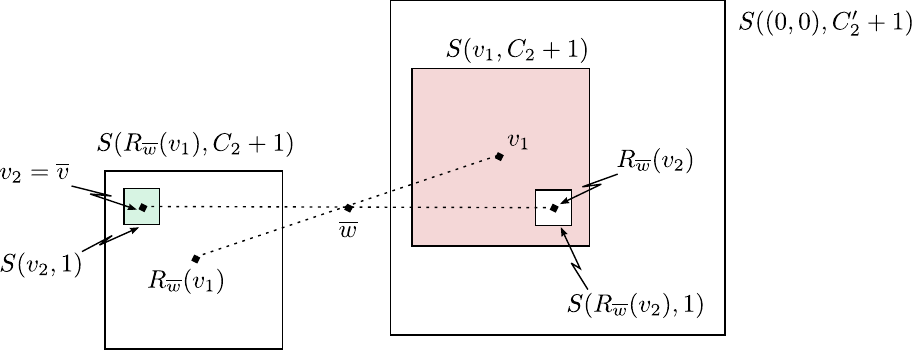}
\caption{}
\label{fig:claim2}
\end{figure}
\endproof

To complete the proof of theorem \ref{th:maintheorem}, suppose, suppose that $\hat f^{n}(O)\cap T_{\ol{w}}(O)= \emptyset$ for all $n\in \Z$.
Claims \ref{cl:mainclaim} and \ref{cl:pintaquadradoeoposto} together imply that, for sufficiently large $n,$  both $\ol{v} \in I(n,(0,0))$ and $2\ol{w}- \ol{v} =R_{\ol{w}}(\ol{v}) \in I(n,(0,0))$. This means that  $\hat f^{n}(O) \cap  T_{\ol{v}}(O)\neq \emptyset$ and $\hat f^{n}(O) \cap  T_{2\ol{w}-\ol{v}}(O) \neq \emptyset$. Letting $K_1= T_{\ol{w}}(O)$ and $K_2= \hat f^{n}(O)$, it follows that $T_{\ol{w}-\ol{v}}^{-1}(K_1)\cap K_2\neq \emptyset$ and $T_{\ol{w}-\ol{v}}(K_1)\cap K_2\neq \emptyset$. Thus, lemma \ref{lm:inbetween} (with $v=\ol{w}-\ol{v}$) implies that $K_1$ intersects $K_2$, \ie $\hat f^{n}(O)\cap T_{\ol{w}}(O)\neq \emptyset$. This contradicts our assumption at the beginning of this paragraph, completing the proof. \qed

\bibliographystyle{amsalpha}

\end{document}